\documentclass[12pt,oneside]{amsart}
\usepackage{amssymb,verbatim,enumerate,ifthen}
\usepackage{cite}
\usepackage[mathscr]{eucal}
\usepackage[utf8]{inputenc}
\usepackage[T1]{fontenc}
\usepackage[marginparwidth=2.4cm]{geometry}
\newtheorem{thm}{Theorem}[section]
\newtheorem{cor}[thm]{Corollary}
\newtheorem{prop}[thm]{Proposition}
\newtheorem{lem}[thm]{Lemma}

\newtheorem*{claim}{Claim}

\theoremstyle{definition}

\newtheorem{exa}[thm]{Example}

\newtheorem*{xrem}{Remark}

\numberwithin{equation}{section}
\def\eq#1{{\rm(\ref{#1})}}
\def\Eq#1#2{\ifthenelse{\equal{#1}{*}}
  {\begin{equation*}\begin{aligned}[]#2\end{aligned}\end{equation*}}
  {\begin{equation}\begin{aligned}[]\label{#1}#2\end{aligned}\end{equation}}}

\def\A{\mathscr{A}}
\def\B{\mathscr{B}}
\def\D{\mathscr{D}}
\def\E{\mathscr{E}}
\def\G{\mathscr{G}}
\def\M{\mathscr{M}}

\newcommand\R{\mathbb{R}}

\newcommand\N{\mathbb{N}}

\newcommand\WD{\mathscr{W}(D)}
\newcommand\WS{\mathscr{W}(S)}

\DeclareMathOperator{\sign}{sign}

\DeclareMathOperator{\conv}{conv}
\DeclareMathOperator{\cone}{cone}

\title
{Decision making via generalized Bajraktarevi\'c means}

\author{Zsolt P\'ales}
\address{Institute of Mathematics, University of Debrecen, Pf.\ 400, 4002 Debrecen, Hungary}
\email{pales@science.unideb.hu}

\author{Pawe\l{} Pasteczka}
\address{Institute of Mathematics, Pedagogical University of Krak\'ow, Podchor\k{a}\.{z}ych str 2, 30-084 Krak\'ow, Poland}
\email{pawel.pasteczka@up.krakow.pl}

\thanks{The research of the first author was supported by the K-134191 NKFIH Grant and by the 2019-2.1.11-T\'ET-2019-00049, EFOP-3.6.1-16-2016-00022, EFOP-3.6.2-16-2017-00015 projects. The last two projects are co-financed by the European Union and the European Social Fund.}

\keywords{Decision making function, aggregation function, effort functions, synergy, generalized Bajraktarevi\'c mean, equality problem}

\subjclass[2010]{Primary: 90B50; Secondary: 26D15, 26E60, 39B62}

\begin{document}
\begin{abstract}
We define decision-making functions which arise from studying the multidimensional generalization of the weighted Bajraktarević means. It allows a nonlinear approach to optimization problems.

These functions admit several interesting (from the point of view of decision-making) properties, for example, delegativity (which states that each subgroup of decision-makers can aggregate their decisions and efforts), casuativity (each decision affects the final outcome except two trivial cases) and convexity-type properties. 

Beyond establishing the most important properties of such means, we solve their equality problem, we introduce a notion of synergy and characterize the null-synergy decision-making functions of this type.
\end{abstract}
\maketitle

\section{Introduction}

In game theory the mathematical models for decision-making create challenging and important problems binding computer science, economy, mathematics and psychology. In one of such models there is a set $D$ of all possible decisions and a finite number $n$ of players (decision makers) with their individual nonnegative weights (efforts) and decisions. Obviously, the sum of all weights cannot be zero and thus (as a vector) it belongs to the set
\Eq{*}{
W_n:=[0,+\infty)^n \setminus \{(0,\dots,0)\}.
}
The issue is to aggregate all the individual decisions with the corresponding weights to one (common) decision. For this purpose, we need the notion of an \emph{aggregation function} on $D$, which is defined to be a map
\Eq{*}{
\M\colon \WD \to D,
\qquad\mbox{where}\qquad\WD:=\bigcup_{n=1}^\infty D^n \times W_n.
}

For instance, when the set $D$ of decisions is a convex subset of a linear space $X$, then the weighted arithmetic mean $\A$, which is defined as
\Eq{*}{
  \A(x,\lambda)
  :=\frac{\lambda_1x_1+\dots+\lambda_nx_n}{\lambda_1+\dots+\lambda_n}
  \qquad(x=(x_1,\dots,x_n)\in D^n,\, 
     \lambda=(\lambda_1,\dots,\lambda_n)\in W_n),
}
is a well-known aggregation function. Further examples for an aggregation function are as follows (see for example \cite{Hen98}):
\begin{enumerate}
\item The \emph{Primacy Effect} $\D_{PE} \colon \WD \to D$ is defined by
\Eq{*}{
\D_{PE}(x,\lambda):=\big\{x_i\:\big|\ \lambda_i \ne 0 \text{ and }\lambda_j =0\text{ for all } j\in\{1,\dots,i-1\} \big\}.
}
\item The \emph{Recency Effect} $\D_{RE} \colon \WD \to D$ is defined by
\Eq{*}{
\D_{RE}(x,\lambda):=\big\{x_i\:\big|\ \lambda_i \ne 0 \text{ and }\lambda_j =0\text{ for all }j\in\{i+1,\dots,n\}\big\}.
}
 \item The \emph{First Dominating Decision} $\D_{FDD} \colon \WD \to D$ is given by
\Eq{*}{
\D_{FDD}(x,\lambda):=\big\{x_i\:\big|\ \lambda_i =\max(\lambda) \text{ and }\lambda_j <\lambda_i \text{ for all }j\in\{1,\dots,i-1\}\big\}.
}
\item The \emph{First Dominant} $\D_{FD} \colon \WD \to D$ is given by
\Eq{*}{
\D_{FD}(x,\lambda):=\D_{FDD}(x,\lambda^*),
\qquad\mbox{where }\lambda_i^*:=\sum_{j\colon x_j=x_i} \lambda_j.
}
\end{enumerate}
All functions listed in (1)--(4) are reflexive, eliminative, nullhomogeneous in the weights but not symmetric (see the relevant definitions below). Furthermore, they are all \emph{conservative} (or \emph{selective}), which means that the aggregated decision is always one of the individual ones. For a detailed study of (nonweighted) conservative aggregation functions, we refer the reader to the recent study by Couceiro--Devillet--Marichal \cite{CouDevMar18} and Devillet--Kiss--Marichal \cite{DevKisMar19}.

In many settings, $D$ is an infinite set which often refers to the position of the players in a space before the game. An aggregation function unites the positions of all the players into one. An individual nonnegative weight measures the impact of the decision of the corresponding players to the final outcome. In order to introduce plausible and natural properties for aggregation functions, we introduce the concept of decision-making functions on an arbitrary set $D$. For this aim, we adopt the notion of weighted means (which were defined on an interval) from the paper \cite{PalPas18b} to our more general setting. An aggregation function $\M \colon \WD \to D$ is called a \emph{decision-making function (on $D$)} if it satisfies the following five conditions: 
 \begin{enumerate}[(i)]
  \item $\M$ is \emph{reflexive}: For all $x \in D$ and $\lambda \in \R_+$, 
  we have $\M(x,\lambda)=x$.
  \item $\M$ is \emph{nullhomogeneous in the weights}: For all $n \in \N$, $(x_1,\dots,x_n)\in D^n$, $(\lambda_1,\dots,\lambda_n)\in W_n$, and $t \in \R_+$, we have  
  \Eq{*}{\qquad
   \M\big((x_1,\dots,x_n),(t\lambda_1,\dots,t\lambda_n)\big)
   =\M\big((x_1,\dots,x_n),(\lambda_1,\dots,\lambda_n)\big).
  }
  \item $\M$ is \emph{symmetric}: For all $n \in \N$, $(x_1,\dots,x_n)\in D^n$, $(\lambda_1,\dots,\lambda_n)\in W_n$ and for all permutations $\sigma$ of $\{1,\dots,n\}$, we have 
  \Eq{*}{\qquad
  \M\big((x_{\sigma(1)},\dots,x_{\sigma(n)}),
       (\lambda_{\sigma(1)},\dots,\lambda_{\sigma(n)})\big)
    =\M\big((x_1,\dots,x_n),(\lambda_1,\dots,\lambda_n)\big).
  }
  \item $\M$ is \emph{eliminative or neglective}: For all $n\geq 2$, $(x_1,\dots,x_n)\in D^n$ and $(\lambda_1,\dots,\lambda_n)\in W_n$ with $\lambda_1=0$, we have
  \Eq{*}{\qquad
  \M\big((x_1,\dots,x_n),(\lambda_1,\dots,\lambda_n)\big)
  =\M\big((x_2,\dots,x_n),(\lambda_2,\dots,\lambda_n)\big).
  }
   \item $\M$ is \emph{reductive}: For all $n\geq 2$, $(x_1,\dots,x_n)\in D^n$ with $x_1=x_2$ and $(\lambda_1,\dots,\lambda_n)\in W_n$, we have
  \Eq{*}{\qquad
  \M\big((x_1,\dots,x_n),(\lambda_1,\dots,\lambda_n)\big)
  =\M\big((x_2,x_3,\dots,x_n),(\lambda_1+\lambda_2,\lambda_3,\dots,\lambda_n)\big).
  }
\end{enumerate}

We also introduce the concept of the effort function, which is aiming to aggregate the individual weights (efforts) into one positive number: A function $\E \colon \WD \to \R_+$ is called an \emph{effort function (on $D$)} if it satisfies the following five conditions: 
 \begin{enumerate}[(i)]
  \item $\E$ is \emph{reflexive in the weights}: For all $x \in D$ and $\lambda \in \R_+$, we have $\E(x,\lambda)=\lambda$.
  \item $\E$ is \emph{homogeneous in the weights}: For all $n \in \N$, $(x_1,\dots,x_n)\in D^n$, $(\lambda_1,\dots,\lambda_n)\in W_n$, and $t \in \R_+$, we have  
  \Eq{*}{\qquad
   \E\big((x_1,\dots,x_n),(t\lambda_1,\dots,t\lambda_n)\big)
   =t\E\big((x_1,\dots,x_n),(\lambda_1,\dots,\lambda_n)\big).
  }
  \item $\E$ is \emph{symmetric}.
  \item $\E$ is \emph{eliminative or neglective}. 
  \item $\E$ is \emph{reductive}. 
\end{enumerate}
One can easily see that the map $\alpha \colon \WD \to \R_+$ given by
\Eq{*}{
  \alpha(x,\lambda)
  :=\lambda_1+\dots+\lambda_n
  \qquad(x\in D^n,\, 
     \lambda=(\lambda_1,\dots,\lambda_n)\in W_n),  
}
is an effort function, which we call the \emph{arithmetic effort function}.

The symmetry property of decision-making and effort functions means that there is no distinction between players and also their order is irrelevant for the decision. This property has a far-reaching consequences especially for conservative functions, as it determines the anty-symmetric preference relation on $D$ by $x\succ y :\!\!\iff x=\M((x,y),(1,1))$ (cf.\ \cite{Dev19} for details).
It was proved experimentally that this relation cannot be generalized to multivariable choice; this phenomena is known as a decoy effect (see for example Huber--Payne--Puto \cite{HubPayPut82}). 

The nullhomogeneity of $\M$ and the homogeneity of $\E$ in the weights states that if the weights are scaled by the same factor, then the decision remains unchanged and the effort is scaled by the same factor. The meaning of the elimination principle is that players with zero weight do not affect the decision and the effort. One can easily check that the arithmetic mean is a decision-making function over any convex subset of a linear space.

We introduce now an aggregation-type property which will play a significant role in the sequel. We say that a decision-making function $\M$ on $D$ is \emph{delegative} (admits the delegation principle or partial aggregation principle) if, for all $(y,\mu)\in \WD$, there exists a pair $(y_0,\mu_0) \in D \times \R_+$ such that
\Eq{ts}{
  \M((x,y),(\lambda,\mu))=\M((x,y_0),(\lambda,\mu_0))  
  \qquad((x,\lambda)\in \WD).
}
Analogously, we can speak about the delgativity of an effort function $\E$ on $D$ which means that, for all $(y,\mu)\in \WD$, there exists a pair $(y_0,\mu_0) \in D \times \R_+$ such that
\Eq{ts+}{
  \E((x,y),(\lambda,\mu))=\E((x,y_0),(\lambda,\mu_0))  
  \qquad((x,\lambda)\in \WD).
}

\begin{lem} Let $(y,\mu)\in \WD$ be fixed. If $\M\colon\WD \to D$ is a delegative decision-making function, then \eq{ts} holds if and only if $y_0=\M(y,\mu)$. Analogously, if $\E\colon\WD \to \R_+$ is a delegative effort function, then \eq{ts+} holds if and only if $\mu_0=\E(y,\mu)$.
\end{lem}

\begin{proof} Using the properties of decision making functions and applying the delegativity of $\M$ for $(x,\lambda)=(y,\mu)$ twice, we get
\Eq{*}{
  \M(y,\mu)=\M((y,y),(\mu,\mu))
  =\M((y,y_0),(\mu,\mu_0))=\M((y_0,y_0),(\mu_0,\mu_0))
  =y_0.
}
Similarly, the properties of effort functions and applying the delegativity
of $\E$ yield 
\Eq{*}{
  \E(y,\mu)=\frac12\E((y,y),(\mu,\mu))
  =\frac12\E((y,y_0),(\mu,\mu_0))=\frac12\E((y_0,y_0),(\mu_0,\mu_0))
  =\mu_0,
}
which completes the proof of the lemma.
\end{proof}

Motivated by the above statement, a delegative decision-making function $\M$ and a delegative effort function $\E$ are called \emph{associated} if, for all $(y,\mu)\in\WD$, the equalities \eq{ts} and \eq{ts+} hold with $(y_0,\mu_0)=(\M(y,\mu),\E(y,\mu))$.

The aim of this paper is to present a construction of a broad class of decision-making and effort functions which arise from studying so-called weighted Bajraktarević means. Despite of the analytical background of this paper, we are convinced that our construction could provide useful models for game theoretical and for decision-making problems.

Our motivation is to present a mean-type approach to studying the Farm Structure Optimization Problem -- cf. for example 
Abd El-Wahed--Abo-Sinna \cite{AbdAbo01}, Czy\.{z}ak \cite{Czy90}, S\l{}owi\'{n}ski--Teghem \cite{SloTeg90}, Tzeng--Huang \cite{HuaTze14}, Xu--Zhou \cite{XuZho11}.

\section{Observability and conical convexity}

A subset $S$ of a linear space $X$ over $\R$ is called a \emph{ray} if $S=R_x:=\R_+x:=\{\lambda x\mid\lambda\in\R_+\}$ holds for some nonzero element $x\in X$. If $y\in R_x$, then the unique positive numbers $\lambda$ for which $y=\lambda x$ holds will be denoted by $[y:x]$. 

A set $S$ is called a \emph{cone} if it is the union of rays of $X$. One can see that $S$ is a cone precisely if it is closed under multiplication by positive scalars. The cone generated by the set $S$ -- denoted by $\cone(S)$ -- is the smallest cone containing $S$. It is clear that $\cone(S)=\R_+S=\bigcup_{x\in S}R_x$. 

A subset $S$ of $X$ over $\R$ is called \emph{observable (from the origin)} if the rays generated by two distinct elements of $S$ are disjoint. Note that for an observable subset $S$, every ray contained in $\cone(S)$ intersects $S$ at exactly one point. Moreover, the family $\{R_x \mid x \in S\}$ is a partition of $\cone(S)$. Due to this fact, for every observable set $S$ one can define the projection along rays (briefly the \emph{ray projection}) $\pi_S \colon \cone(S) \to S$ as follows: For every $x\in S$ and $y\in R_x$, we have $\pi_S(y)=x$. As a matter of fact, observability of $S$ is not only sufficient but also necessary to define such map.  

We say that a function $f\colon D\to X$ is \emph{observable} if it is injective and has an observable image. Analogously to the previous setup, we define $\pi_f:=\pi_{f(D)}$. The extended inverse of $f$, denoted by $f^{(-1)}\colon \cone (f(D))\to D$ is defined by 
\Eq{*}{
  f^{(-1)}:=f^{-1}\circ \pi_f.
}
Clearly, $f^{(-1)}(x)=f^{-1}(x)$ if $x\in f(D)$.

We say that a subset $S$ of $X$ is \emph{conically convex} if $\cone(S)$ is a convex set and convex hull of $S$ does not contain the origin.
Note that conical convexity is a weaker property than convexity: every convex set which does not contain zero is conically convex, the converse implication is not true. 

Hereafter an observable function $f \colon D \to X$ such that $f(D)$ is conically convex is called \emph{admissible}.

In the next result we give a characterization of admissibility for functions with 2-dimensional range. 

\begin{thm}\label{thm:det}
Let $f=(f_1,f_2):I\to\R^2$ be continuous. Then $f$ is admissible if and only if, for all distinct elements $x,y\in I$, $f(x)$ and $f(y)$ are linearly independent, that is,
\Eq{det}{
  \left|\begin{matrix}
         f_1(x) & f_1(y)\\f_2(x) & f_2(y)
  \end{matrix}\right|
  \neq 0.  
}
\end{thm}

\begin{proof}
Assume first that $f$ is admissible but \eq{det} is not valid for some $x,y\in I$ with $x\neq y$. Then the two vectors $f(x)$ and $f(y)$ are linearly dependent. Thus, there exist $(\alpha,\beta)\neq(0,0)$ such that $\alpha f(x)+\beta f(y)=0$.
Because $(0,0)$ is not contained in $f(I)$, we have that $f(x)\neq (0,0)\neq f(y)$, hence $\alpha\beta\neq0$. In the case when $\alpha\beta<0$, we get that the rays generated by $f(x)$ and $f(y)$ are the same, which contradicts the observability. If $\alpha\beta>0$, then $(0,0)$ is in the segment connecting $[f(x),f(y)]$, which contradicts the property that the convex hull of $f(I)$ does not contain the origin.

Assume now that \eq{det} holds for all $(x,y)\in \Delta(I):=\{(x,y)\mid x,y\in I,\,x<y\}$. This immediately shows that $f(I)$ is observable. 

Since $\Delta(I)$ is a convex set and $f$ is continuous, therefore the determinant is either positive on $\Delta(I)$ or negative on $\Delta(I)$. We may assume that it is positive everywhere on $\Delta(I)$.
Then the pair $(f_1,f_2)$ is a so-called Chebyshev system on $I$. According to \cite[Theorem 2]{BesPal03}, there exist two real constants $\alpha$ and $\beta$ such that $\alpha f_1+\beta f_2$ is positive over the interior of $I$ and thus it is nonnegative over $I$. This implies that the curve $f(I)$ is contained in the closed half plain $P=\{(u,v)\in\R^2\mid \alpha u+\beta v\geq0\}$. Consider now the unit circle and project the curve $f(I)$ into it. Then, by the continuity of the projection mapping, the projection of the curve is an arc of the unit circle which is contained in $P$. It is obvious that the conical hull of the curve $f(I)$ and the conical hull of the arc are identical. On the other hand, the property that the arc is contained in the half plain $P$, implies that its conical hull is convex. This shows the conical hull of $f(I)$ is also convex.

The origin cannot be contained in the conical hull of $f(I)$ because then there were two distinct points of $f(I)$ belonging to the boundary line $\{(u,v)\in\R^2\mid \alpha u+\beta v=0\}$ such that the origin were in the segment connecting them. This contradicts the property that every two distinct points
of $f(I)$ are linearly independent.
\end{proof}

\section{Generalized Bajraktarevi\'c means and their properties}

Given an admissible function $f\colon D\to X$, we can define the \emph{(weighted) generalized Bajraktarevi\'c mean} $\B_f\colon\WD\to D$ by
\Eq{*}{
\B_f(x,\lambda)=f^{(-1)} \Big( \sum_{k=1}^n \lambda_k f(x_k) \Big),
}
where $x =(x_1,\dots,x_n)\in D^n$ are the entries and $\lambda=(\lambda_1,\dots,\lambda_n) \in W_n$ are the corresponding weights. We also define $\beta_f \colon \WD\to \R_+$ by
\Eq{*}{
  \beta_f(x,\lambda):=\bigg[\sum_{k=1}^n \lambda_k f(x_k):f(\B_f(x,\lambda))\bigg],
}
which we call the \emph{effort function associated to $\B_f$}.

Intuitively $D$ is the space of all possible decisions, $x$ is a vector of all considered decisions, and $\lambda$ is the effort of players invested in the option. In this way $\B_f(x,\lambda)$ is the decision derived from the possible decisions and their weights.

\begin{lem}\label{lem:afy}
Let $f\colon D \to X$ be an admissible function. Then, for all $n \in \N$ and $(x,\lambda) \in D^n \times W_n$, there exists exactly one pair $(u,\eta) \in D \times \R_+$ satisfying 
\Eq{E:afy}{
\eta f(u)= \sum_{k=1}^n \lambda_k f(x_k).
}
Furthermore, $u=\B_f(x,\lambda)$ and $\eta=\beta_f(x,\lambda)$. 
\end{lem}
 
\begin{proof}
As $f(D)$ is conically convex, we know that 
\Eq{*}{
\frac{\sum_{k=1}^n \lambda_k f(x_k)}{\sum_{k=1}^n \lambda_k} \in \cone f(D).
}
Thus there exists $\eta \in \R_+$ such that 
\Eq{*}{
  \eta^{-1} \sum_{k=1}^n \lambda_k f(x_k) \in f(D).
}
Moreover, as $f(D)$ is observable, the constant $\eta$ is uniquely determined by $x$ and $\lambda$. Now, as $f$ is injective, there exists exactly one $u\in D$ such that 
\Eq{*}{
  \eta^{-1} \sum_{k=1}^n \lambda_k f(x_k)=f(u),
}
which is trivially equivalent to \eq{E:afy}. The last assertion simply follows from the definition of $\B_f(x,\lambda)$ and $\beta_f(x,\lambda)$.
\end{proof}

Now $\B_f$ is the equivalent (or aggregated) decision, $\beta_f$ is the equivalent (or aggregated) effort. In fact $\beta_f$, can be considered as the amount of goods we need to invest into a single decision to be irrelevant between the diversed an the aggregated situation. This property is much more transparent in view of Theorem~\ref{thm:subs}.

\begin{thm} 
Let $f\colon D\to X$ be an admissible function. Then $\B_f$ is a decision-making function and $\beta_f$ is an effort function on $D$.
\end{thm}

\begin{proof}
Fix $n \in \N$ and $(x,\lambda) \in D^n \times W_n$. In view of Lemma~\ref{lem:afy} both $\B_f(x,\lambda)$ and $\beta_f(x,\lambda)$ depends on its arguments implicitly via the sum $\sum_{k=1}^n \lambda_k f(x_k)$. This immediately implies that $\B_f$ and $\beta_f$ are both symmetric, reductive and eliminative. The nullhomogeneity of $\B_f$ is obvious in view of the definition of $f^{(-1)}=f^{-1} \circ \pi_f$. The reflexivity properties of $\B_f$ and $\beta_f$ are immediate consequences of the equalities 
\Eq{*}{
\B_f(y,\lambda)&=f^{(-1)}(\lambda f(y))=f^{-1}\circ f(y)=y,\\
\beta_f(y,\lambda)
&=[\lambda f(y)\colon f(\B_f(y,\lambda))]
=[\lambda f(y)\colon f(y)]=\lambda\qquad (y \in D,\lambda\in\R_+).
}

Now we only need to verify the homogeneity of $\beta_f$ in the weights. To this end, take additionally $t \in \R_+$. 
By the definition of $\beta_f$ and the nullhomogeneity of $\B_f$, we have
\Eq{*}{
\beta_f(x,t\lambda)&=\bigg[\sum_{k=1}^n t\lambda_k f(x_k):f(\B_f(x,t\lambda))\bigg]\\
&=t\bigg[\sum_{k=1}^n \lambda_k f(x_k):f(\B_f(x,\lambda))\bigg]=t\beta_f(x,\lambda),
}
which completes the proof.
\end{proof}

\begin{thm}\label{thm:f1f2}
Let $f_1,f_2\colon I \to \R$ be continuous functions such that $f_2$ is nowhere zero and $f_1/f_2$ is strictly monotone. Then $f=(f_1,f_2)$ is admissible and we have the equalities
\begin{align}
\B_f(x,\lambda)&=\Big(\frac{f_1}{f_2}\Big)^{-1} \bigg(\frac{\lambda_1 f_1(x_1)+\dots+\lambda_nf_1(x_n)}{\lambda_1 f_2(x_1)+\dots+\lambda_nf_2(x_n)} \bigg) &\qquad ((x,\lambda)\in I^n \times W_n), \label{Bf}\\
\beta_f(x,\lambda)
  &=\frac{\lambda_1 f_2(x_1)+\dots+\lambda_nf_2(x_n)}
  {f_2(\B_f(x,\lambda))} &\qquad ((x,\lambda)\in I^n \times W_n)\label{bf}. 
\end{align}

\end{thm}

\begin{proof} First we show that $f=(f_1,f_2)$ is admissible. In view of Theorem~\ref{thm:det}, it is sufficient to show that \eq{det} holds for all distinct elements $x,y$ of $I$. Indeed, 
\Eq{*}{
  \left|\begin{matrix}
         f_1(x) & f_1(y)\\f_2(x) & f_2(y)
  \end{matrix}\right|
  =f_1(x)f_2(y)-f_1(y)f_2(x)
  =f_2(x)f_2(y)\bigg(\frac{f_1(x)}{f_2(x)}-\frac{f_1(y)}{f_2(y)}\bigg),
}
which is nonzero by our assumptions.

To prove the identity \eq{Bf},
let $n \in \N$ and a pair $(x,\lambda)\in I^n \times W_n$ be fixed. As $f(\B_f(x,\lambda))$ and $\sum_{i=1}^n \lambda_i f(x_i)$ are on the same ray, we have
\Eq{*}{
\frac{f_1(\B_f(x,\lambda))}{f_2(\B_f(x,\lambda))}= \frac{\lambda_1 f_1(x_1)+\dots+\lambda_nf_1(x_n)}{\lambda_1 f_2(x_1)+\dots+\lambda_nf_2(x_n)}.
}
Using that $f_1/f_2$ is invertible and applying its inverse side-by-side, we obtain the desired equality \eq{Bf}. The equality 
\Eq{*}{
\beta_f(x,\lambda)
=\bigg[\sum_{k=1}^n \lambda_k f(x_k):f(\B_f(x,\lambda))\bigg]
=\bigg[\sum_{k=1}^n \lambda_k f_2(x_k):f_2(\B_f(x,\lambda))\bigg]
}
implies \eq{bf} immediately.
\end{proof}

A classical particular case of the above theorem is when $f_1$ and $f_2$ are power functions on $I=\R_+$: $f_1(x)=x^p$ and $f_2(x)=x^q$, where $p,q\in\R$ with $p\neq q$. Then
\Eq{E:Ginipq}{
  \B_f(x,\lambda)&=G_{p,q}(x,\lambda)
  :=\bigg(\frac{\lambda_1 x_1^p+\dots+\lambda_n x_n^p}
  {\lambda_1 x_1^q+\dots+\lambda_n x_n^q}\bigg)^{\frac{1}{p-q}},\\
  \beta_f(x,\lambda)&=\gamma_{p,q}(x,\lambda)
  :=\frac{(\lambda_1 x_1^q+\dots+\lambda_n x_n^q)^{\frac{p}{p-q}}}
  {(\lambda_1 x_1^p+\dots+\lambda_n x_n^p)^{\frac{q}{p-q}}},
}
which are called the \emph{Gini mean} and the \emph{Gini effort function} of parameter $(p,q)$ where $p\neq q$ (cf.\ \cite{Gin38}). For the case $p=q$, let $f_1(x):=x^p\ln(x)$ and $f_2(x):=x^p$. Then we have
\Eq{*}{
  \B_f(x,\lambda)&=G_{p,p}(x,\lambda)
  :=\exp\bigg(\frac{\lambda_1 x_1^p\ln x_1+\dots+\lambda_n x_n^p\ln x_p}
  {\lambda_1 x_1^p+\dots+\lambda_n x_n^p}\bigg),\\
  \beta_f(x,\lambda)&=\gamma_{p,p}(x,\lambda)
  :=(\lambda_1 x_1^p+\dots+\lambda_n x_n^p) \exp\bigg(( -p) \cdot \frac{\lambda_1 x_1^p\ln x_1+\dots+\lambda_n x_n^p\ln x_p}
  {\lambda_1 x_1^p+\dots+\lambda_n x_n^p}\bigg),
}

\section{Aggregation-type properties}

\subsection{Delegativity}
In this subsection we establish the delegativity of the generalized Bajraktarević means and the corresponding effort functions, moreover, we show that these two maps are associated to each other. 

\begin{thm}\label{thm:subs} 
Let $f \colon D \to X$ be an admissible function. 
Then both $\B_f$ and $\beta_f$ are delegative and associated.
\end{thm}

\begin{proof} Fix $m\in\N$ and $(y,\mu)\in D^m\times W_m$. For an arbitrary $(x,\lambda)\in D^n\times W_n$, according to the Lemma~\ref{lem:afy}, we have 
\Eq{*}{
  \beta_f((x,y),(\lambda,\mu)) f(\B_f((x,y),(\lambda,\mu)))
  &=\sum_{k=1}^n \lambda_k f(x_k)+\sum_{j=1}^m \mu_jf(y_j),\\
  \qquad \beta_f(y,\mu) f(\B_f(y,\mu))&=\sum_{j=1}^m \mu_jf(y_j).
}
Then, these equalites and Lemma~\ref{lem:afy} imply
\Eq{*}{
  \beta_f((x,y),(\lambda,\mu)) f(\B_f((x,y),(\lambda,\mu)))
  &=\sum_{k=1}^n \lambda_k f(x_k)+\beta_f(y,\mu) f(\B_f(y,\mu))\\
  &=\beta_f((x,y_0),(\lambda,\mu_0))f(\B_f((x,y_0),(\lambda,\mu_0))),
}
which then yields \eq{ts}.
\end{proof}

The following result could be derived from the above theorem, but we shall provide a direct and short proof for it.

\begin{cor} 
Let $f \colon D \to X$ be an admissible function. Let $n,m\in\N$, $x\in D^n$ and $\lambda^{(1)},\dots,\lambda^{(m)}\in W_n$. Denote $y_j:=\B_f(x,\lambda^{(j)})$ and $\mu_j:=\beta_f(x,\lambda^{(j)})$ for $j\in\{1,\dots,m\}$. Then, for all $(t_1,\dots,t_m)\in W_m$, 
\Eq{tt}{
  \B_f(x,t_1\lambda^{(1)}+\dots+t_m\lambda^{(m)})
  &=\B_f(y,(t_1\mu_1,\dots,t_m\mu_m)),\\
  \beta_f(x,t_1\lambda^{(1)}+\dots+t_m\lambda^{(m)})
  &=\beta_f(y,(t_1\mu_1,\dots,t_m\mu_m)).
}
\end{cor}

\begin{proof} By the definitions of $y_1,\dots,y_m$ and $\mu_1,\dots,\mu_m$, according to Lemma~\ref{lem:afy}, we have
\Eq{*}{
\mu_jf(y_j)
  =\beta_f(x,\lambda^{(j)})f(\B_f(x,\lambda^{(j)}))
  =\sum_{i=1}^n \lambda^{(j)}_i f(x_i) \qquad (j\in\{1,\dots,m\}).
}
Multiplying this equality by $t_j$ side by side, and then summing up the equalities so obtained for $j\in\{1,\dots,m\}$, we get
\Eq{*}{
\beta_f&(y,(t_1\mu_1,\dots,t_m\mu_m))
  f(\B_f(y,(t_1\mu_1,\dots,t_m\mu_m)))\\
  &=\sum_{j=1}^m t_j\mu_jf(y_j)
  =\sum_{j=1}^m t_j\bigg(\sum_{i=1}^n \lambda^{(j)}_if(x_i)\bigg)   =\sum_{i=1}^n \bigg(\sum_{j=1}^m t_j\lambda^{(j)}_i\bigg) f(x_i)\\
  &=\beta_f(x,t_1\lambda^{(1)}+\dots+t_m\lambda^{(m)})
  f(\B_f(x,t_1\lambda^{(1)}+\dots+t_m\lambda^{(m)})).
}
The above equality, by the observability of $f$, yields \eq{tt}.
\end{proof}

The latter corollary can be also rewritten in the matrix form.

\begin{cor}\label{cor:subs} 
Let $f \colon D \to X$ be an admissible function. Let $x\in D^n$, and
$\Lambda=(\lambda^{(1)},\dots,\lambda^{(m)})\in W_n^m$. Define a vector $y:=\big(\B_f(x,\lambda^{(i)})\big)_{i=1}^m \in D^m$ and $\mu:=\big(\beta_f(x,\lambda^{(i)})\big)_{i=1}^m \in \R_+^m$. Then,
for all $t\in W_m$,
\Eq{E:Lt}{
  \B_f(x,\Lambda t)=B_f(y,\mu\!\cdot\! t)\quad\text{ and }\quad \beta_f(x,\Lambda t)=\beta_f(y,\mu\!\cdot\! t),
}
where ``$\cdot$'' stands for the coordinate-wise multiplication of the elements of $W_m$. 
\end{cor}

\subsection{Casuativity} 
We say that a decision making function $\M$ on $D$ is \emph{casuative} if for all $(x,\lambda)\in\WD$ and for all pair $(y,\mu) \in D \times (0,+\infty)$, we have
\Eq{E:casu}{
\M(x,\lambda)=\M\big((x,y),(\lambda,\mu)\big) \iff y=\M(x,\lambda).
}

Casuativity is somehow opposite to conservativity. Namely there holds the following easy-to-see lemma.
\begin{lem}
 Let $\M$ be a symmetric and casuative decision-making function on $D$. Then for all distinct $x_1,x_2 \in D$ and $\lambda_1,\lambda_2 \in (0,+\infty)$ we have $\M((x_1,x_2),(\lambda_1,\lambda_2)) \in D \setminus\{x_1,x_2\}$.
\end{lem}

It is also reasonable to define \emph{weak casuativity} in the case when the $(\Leftarrow)$ implication of \eq{E:casu} holds. Observe that there are number of weakly casuative decision making functions which are not casuative, for example decision making functions which are induced by a preference relation. 

\begin{prop}
 Generalized Bajraktarević mean $\B_f$ is casuative for every admissible function $f \colon D \to X$.
\end{prop}
\begin{proof}
As $\B_f$ is delegative it is sufficient to show that \eq{E:casu} holds for all $(x,\lambda)\in D \times \R_+$. The $(\Leftarrow)$ part is then trivial. To prove the converse implication let $x,y \in D$ and $\lambda, \mu \in (0,+\infty)$ such that
\Eq{*}{
\B_f\big((x,y),(\lambda,\mu)\big)=\B_f(x,\lambda)=x.
}
Thus, by Lemma~\ref{lem:afy}, there exists $\eta \in \R_+$ such that $\lambda f(x)+\mu f(y)=\eta f(x)$.
Consequently, as $\mu \ne 0$ we get $f(y)=\tfrac{\eta-\lambda}\mu f(x)$. Now admissibility of $f$ implies $0 \notin \conv f(X)$ and therefore $\tfrac{\eta-\lambda}\mu>0$. Then observability of $f$ yields $y=x$.
\end{proof}

This property plays an important role in the definition of effort function.

\begin{cor}\label{cor:Bbeta}
 Let $f\colon D \to X$ and $g \colon D \to Y$ be two admissible function such that $\B_f=\B_g$. Then $\beta_f=\beta_g$.
\end{cor}

\begin{proof}
Let $(x,\lambda)\in\WD$, set $m:=\B_f(x,\lambda)=\B_g(x,\lambda)$ and assume that $d:=\beta_g(x,\lambda)-\beta_f(x,\lambda)\neq 0$. Without loss of generality, we may assume that $d>0$. 

Take $y_0 \in D \setminus \{m\}$ arbitrarily. By Theorem~\ref{thm:subs} and the equality $\B_f=\B_g$ we have
 \Eq{*}{
 \B_f((m,y_0),(\beta_f(x,\lambda),1))
&=\B_f((x,y_0),(\lambda,1))
=\B_g((x,y_0),(\lambda,1))\\
&=\B_g((m,y_0),(\beta_g(x,\lambda),1))
=\B_f((m,y_0),(\beta_g(x,\lambda),1))\\
&=\B_f((m,y_0,m),(\beta_f(x,\lambda),1,d)).
 }
Now by casuativity of $\B_f$ and $d>0$, we obtain
\Eq{*}{
\B_f((m,y_0),(\beta_f(x,\lambda),1))=m,
}
which implies $y_0=m$, a contradiction. Therefore, $d=0$, i.e., $\beta_g(x,\lambda)=\beta_f(x,\lambda)$.
\end{proof}

\section{Convexity induced by admissible functions}

For an admissible function $f \colon D \to X$, a subset $S\subseteq D$ is called \emph{$f$-convex} if, for all $n\in\N$, $(x,\lambda)\in S^n\times W_n$, we have that $\B_f(x,\lambda)\in S$. 

Observe that if $I \subset \R$ is an interval and $f=(f_1,f_2) \colon I \to \R^2$ satisfies conditions of Theorem~\ref{thm:f1f2} then the mean-value property for a Bajraktarević mean can be written as $\B_{f}(x,\lambda)\in J$ for every subinterval $J \subseteq I$, $n \in \N$ and $(x,\lambda)\in J^n \times W_n$. This yields that every subinterval of $I$ is $f$-convex.
Furthermore, in view of continuity of $f$, the converse implication is also valid.

The first assertion characterizes $f$-convexity in terms of standard convexity.

\begin{prop} \label{prop:fconv}
Let $f \colon D \to X$ be an admissible function. Then $S\subseteq D$ is $f$-convex if and only if $\cone(f(S))$ is convex.
\end{prop}

\begin{proof}
Assume first that $S$ is $f$-convex. 
It suffices to show that for every $n \in \N$ and a pair  $(x,\lambda) \in S^n \times W_n$ we have $\xi:=\sum_{i=1}^n \lambda_if(x_i) \in \cone (f(S))$. 
But $\pi_f(\xi)=f(\B_f(x,\lambda))$, thus $\xi \in R_{f(\B_f(x,\lambda))}\subseteq \cone (f(S))$.

Conversely, if $\cone(f(S))$ is convex then, as it is positively homogeneous we obtain $r:=\sum_{i=1}^n \lambda_i f(x_i) \in \cone(f(S))$ for every $n \in \N$ and $(x,\lambda)\in S^n \times W_n$. By the definition of projection we obtain $\pi_f(r) \in f(S)$ and therefore $\B_f(x,\lambda)=f^{(-1)}(r)=f^{-1}\circ \pi_f(r) \in S$, which shows that $S$ is $f$-convex.
\end{proof}

Now we show that $f$-convex sets admit two very important properties of convex sets.
\begin{lem} 
Let $f \colon D \to X$ be an admissible function. Then the class of $f$-convex subsets of $D$ is closed with respect to intersection and chain union.
\end{lem}

\begin{proof}
Let $\mathcal{S}$ be an arbitrary family of $f$-convex subsets of $D$. Fix $n \in \N$, and $(x,\lambda) \in \big(\bigcap \mathcal{S}\big)^n \times W_n$. 
Then for every $S \in \mathcal{S}$ we have $x \in S^n$ and, in view \mbox{$f$-convexity} of $S$,  we also have $\B_f(x,\lambda)\in S$.
Therefore $\B_f(x,\lambda) \in \bigcap \mathcal{S}$ which shows that $\bigcap \mathcal{S}$ is $f$-convex.

Now take a chain $\mathcal{Q}$ of $f$-convex sets.
Take $n \in \N$ and $(x,\lambda)\in (\bigcup \mathcal{Q})^n \times W_n$ arbitrarily. Then, by the chain property of $\mathcal{Q}$, there exists $Q \in \mathcal{Q}$ such that $x \in Q^n$. As $Q$ is $f$-convex, we obtain $\B_f(x,\lambda)\in Q\subseteq \bigcup \mathcal{Q}$. Whence $\bigcup \mathcal{Q}$ is $f$-convex subset of $D$, too. 
\end{proof}

Applying this lemma, for every admissible function $f \colon D \to X$ and $S \subset D$ we define 
\emph{$f$-convex hull} of $S$ as the smallest $f$-convex subset of $D$ containing $S$ and denote it by $\conv_f(S)$. 
In the next lemma we show that, similarly to the ordinary convex hull, this definition can be also expressed as a set of all possible combinations of elements in $S$. 

\begin{lem}
Let $f \colon D \to X$ be an admissible function and $S \subseteq D$. Then
\Eq{*}{
\conv_f(S)=\big\{\B_f(x,\lambda)\mid (x,\lambda)\in \WS\big\}.
} 
\end{lem}
\begin{proof}
Denote the set on the right-hand-side of the latter equality by $T$. The inclusion $T \subseteq \conv_f(S)$ is the obvious implication of $f$-convexity of the hull. 
To prove the converse inclusion, we need to show that $T$ is $f$-convex.

Let $m\in \N$, $y=(y_1,\dots,y_m)\in T^m$ and $\nu \in W_m$ be arbitrary. By the definition of $T$, in view of reduction priciple and Corollary~\ref{cor:subs} there exists $n \in \N$, a vector $x \in S^n$ and $\Lambda=(\lambda^{(1)},\dots,\lambda^{(n)}) \in W_n^m$ such that $y_i=\B_f(x,\lambda^{(i)})$ for all $i \in \{1,\dots,m\}$.
Now, Corollary~\ref{cor:subs} implies that there exists $\mu\in \R_+^m$ such that \eq{E:Lt} holds.
In particular for $t:=\big(\tfrac{\nu_i}{\mu_i}\big)_{i=1}^m \in W_m$, we get
\Eq{*}{
\B_f(y,\nu)=\B_f(y,(t_1\mu_1,\dots,t_m\mu_m))=\B_f(x,\Lambda t)\in T,
}
which shows that $T$ is $f$-convex, indeed.  
\end{proof}

We now establish the fundamental relationship between the $f$-convex hull in $D$ and the standard convex hull in $X$.

\begin{prop}
Let $f \colon D \to X$ be an admissible function and $S \subseteq D$. Then 
\Eq{convf}{
\conv_f(S)=f^{(-1)} \big( \conv(f(S)) \big).
}
\end{prop}

\begin{proof}

Observe that for every $y \in \conv_f(S)$ there exists $n \in \N$ and $(x,\lambda)\in S^n \times W_n$ such that $y=\B_f(x,\lambda)$. Then
\Eq{*}{
y=f^{(-1)} \bigg( \sum_{k=1}^n \lambda_kf(x_k) \bigg) \in f^{(-1)} \big( \R_+\conv(f(S)) \big)=
f^{(-1)} \big( \conv(f(S)) \big).
}
Conversely, for every $y \in f^{(-1)} \big( \conv(f(S)) \big)$ there exists $n \in \N$ and $(x,\lambda)\in S^n \times W_n$ with $\sum_{k=1}^n \lambda_k=1$ such that $y=f^{(-1)} \big(\sum_{k=1}^n\lambda_kf(x_k)\big)$. 
Then, by the definition of $\B_f$, we obtain $y=\B_f(x,\lambda)\in\conv_f(S)$ which completes the proof.
\end{proof}

\section{Equality of generalized Bajraktarevi\'c means}

Let us now state our main result which characterizes the equality of two generalized Bajraktarevi\'c means. In the one-dimensional case this is a classical result due to Acz\'el--Dar\'oczy \cite{AczDar63c} and Dar\'oczy--P\'ales \cite{DarPal82}.

\begin{thm}\label{lem:sufequ}
Let $X$ and $Y$ be linear spaces, $D$ be an arbitrary set and $f \colon D \to X$, $g \colon D \to Y$ be admissible functions. Then $\B_f=\B_g$ if and only if $g=A\circ f$ for some linear map $A\colon X\to Y$. 
\end{thm}

\begin{proof}[Unverified alternative proof]
Take $n \in \N$, a pair $(x,\lambda)\in D^n\times W_n$, and denote briefly $y:=\B_f(x,\lambda)\in D$. By Lemma~\ref{lem:afy}, applying $A$ side-by-side in \eq{E:afy} and using the equality $g=A\circ f$ twice, we get
 \Eq{*}{
 \beta_f(x,\lambda) g(y)=A \big( \beta_f(x,\lambda) f(y)\big)=A \bigg( \sum_{k=1}^n \lambda_k f(x_k)\bigg)=
 \sum_{k=1}^n \lambda_k Af(x_k)=\sum_{k=1}^n \lambda_k g(x_k).
 }
Now the converse implication in Lemma~\ref{lem:afy} implies $y=\B_g(x,\lambda)$. As $x$ and $\lambda$ are arbitrary, we obtain the equality $\B_f=\B_g$.

To prove the converse implication assume that $\B_f=\B_g$. Then by Corollary~\ref{cor:Bbeta} we have $\beta_f=\beta_g$.
We shall use to following claim several times in the proof:

\begin{claim}
Let $n \in \N$ and $(x,\lambda)\in D^n\times \R^n$ .
Then 
\Eq{*}{
  \sum_{i=1}^n \lambda_i f(x_i)=0
\qquad \text{
if and only if} \qquad  
  \sum_{i=1}^n \lambda_i g(x_i)=0.
}
\end{claim}
\begin{proof}
For $\lambda\equiv 0$ the statement is trivial. From now on assume that $\lambda$ is a nonzero vector and define
\Eq{la+-}{
\lambda^+:=(\max(0,\lambda_i))_{i=1}^d,\qquad 
\lambda^-:=(\max(0,-\lambda_i))_{i=1}^d. 
}
Then $\lambda^+$ and $\lambda^-$ are disjointly supported with nonnegative entries, and $\lambda=\lambda^+-\lambda^-$. 

By the first equality we have 
\Eq{*}{
  \sum_{i=1}^n \lambda_i^+ f(x_i)=  \sum_{i=1}^n \lambda_i^- f(x_i).
}
Therefore as $\lambda$ is nonzero and $0\notin \conv f(D)$, we obtain that both $\lambda^-,\lambda^+ \in W_n$. Furthermore by the definition
$m:=\B_f(x,\lambda^-)=\B_f(x,\lambda^+)$ and $\mu:=\beta_f(x,\lambda^-)=\beta_f(x,\lambda^+)$. Therefore as $\B_f=\B_g$ and $\beta_f=\beta_g$ we obtain
\Eq{*}{
  \sum_{i=1}^n \lambda_i g(x_i)&=\sum_{i=1}^n \lambda_i^+ g(x_i)-\sum_{i=1}^n \lambda_i^- g(x_i)\\
  &=\beta_g(x,\lambda^+) g(\B_g(x,\lambda^+))-\beta_g(x,\lambda^-) g(\B_g(x,\lambda^-))
  =\mu g(m)-\mu g(m)=0.
}
The second implication is analogous.
\end{proof}

Denote the linear span of $f(D)$ and $g(D)$ by $X_0$ and $Y_0$, respectively. Let $H_f\subseteq f(D)$ be a Hamel base for $X_0$. Then one can choose a system of elements $\{x_\gamma\mid \gamma\in\Gamma\}\subseteq D$ such that $H_f=\{f(x_\gamma)\mid \gamma\in\Gamma\}$. 

We are now going to show that $H_g:=\{g(x_\gamma)\mid \gamma\in\Gamma\}$ is a Hamel base for $Y_0$. Indeed, for every collection of pairwise-distinct elements $\gamma_1,\dots,\gamma_d \in \Gamma$, the system $\{f(x_{\gamma_1}),\dots,f(x_{\gamma_d})\}$ is linearly independent and thus, by our Claim, so is $\{g(x_{\gamma_1}),\dots,g(x_{\gamma_d})\}$. 

To show that it is a Hamel base for $Y_0$, we have to prove that $H_g$ is also a generating system. If not, then there exists an element $x^*\in D$ such that $H_g\cup\{g(x^*)\}$ is linearly independent. Repeating the same argument (by interchanging the roles of $f$ and $g$) it follows that $H_f\cup\{f(x^*)\}$ is linearly independent, which contradicts that $H_f$ is a generating system.

As $H_f$ and $H_g$ are Hamel bases for $X_0$ and $Y_0$, respectively, there exists a unique linear mapping $A \colon X_0\to Y_0$ such that 
\Eq{defA}{
g(x)=Af(x)\qquad \text{for all }x \in D_\Gamma:=\{x_{\gamma}\mid \gamma\in\Gamma\}.
}

Our aim is to extend the latter equality to all $x \in D$. To this end, take $x \in D\setminus D_\Gamma$ arbitrarily. 
Using that $H_f$ is a Hamel base for $X_0$, we can find elements $\gamma_1,\dots,\gamma_d\in \Gamma$ and nonzero real numbers $\lambda_1,\dots,\lambda_d$ such that 
\Eq{*}{
  f(x)=\sum_{i=1}^d \lambda_i f(x_{\gamma_i}).
}
Then $\sum_{i=1}^d \lambda_i f(x_{\gamma_i})-f(x)=0$, and by our claim $\sum_{i=1}^d \lambda_i g(x_{\gamma_i})-g(x)=0$.
Finally we obtain
\Eq{*}{
  g(x)=\sum_{i=1}^d \lambda_i g(x_{\gamma_i})=\sum_{i=1}^d \lambda_i Af(x_{\gamma_i})=A\bigg(\sum_{i=1}^d \lambda_i f(x_{\gamma_i})\bigg)=Af(x).
}
Therefore \eq{defA} holds for all $x \in D$ which completes the proof.
\end{proof}

\section{Synergy}

Before we introduce the notion of synergy let us present some interpretation of the aggregated effort. The initial issue of coalitions in decision making theory (and, more general, theory of cooperation in games) is the problem how to measure the coalition quality. Intiuitively, synergy is the difference between the aggregated effort and the sum of the individual efforts, i.e., the arithmetic effort. For the detailed study of synergy, we refer the reader to \cite{ShuMoj12} and references therein.

\begin{exa}\label{ex:toy}
In a toy model we have three parties in a parliament with a total number of $100$ votes and three parties: Party~A ($\lambda_1$ votes), Party~B ($\lambda_2$ votes), Party~C ($\lambda_3$ votes). Assume that $\lambda_1\ge \lambda_2\ge\lambda_3$.
The are three possible coalitions AB, AC and BC. From the point of view of the dominant decision system (for example $\D_{FDD}$) each coalition above $50$ votes, is equivalent to the same number, the smallest majority which is $51$. Thus we have for all $x \in D^2$ and $i,j \in \{1,2,3\}$ with $i \ne j$,
\Eq{*}{
\alpha(\lambda)=\begin{cases}
                                100 & \sum_{i=1}^n \lambda_i\ge 51;\\
                                \sum_{i=1}^n \lambda_i  & \sum_{i=1}^n \lambda_i\le 50.
                               \end{cases}
} 
Now we could compare the sum of weights with the equivalent weight in two cases:

\noindent {\it Situation I:} $(\lambda_1,\lambda_2,\lambda_3)=(45,35,20)$:
\Eq{*}{
s_{AB}&=\alpha(\lambda_1,\lambda_2)-(\alpha(\lambda_1)+\alpha(\lambda_2))=\lambda_3=20;\\
s_{AC}&=\alpha(\lambda_1,\lambda_3)-(\alpha(\lambda_1)+\alpha(\lambda_3))=\lambda_2=35;\\
s_{BC}&=\alpha(\lambda_2,\lambda_3)-(\alpha(\lambda_2)+\alpha(\lambda_3))=\lambda_1=45;\\
s_{ABC}&=\alpha(\lambda_1,\lambda_2,\lambda_3)-(\alpha(\lambda_1)+\alpha(\lambda_2)+\alpha(\lambda_3))&=0.
}
Obviously, each party wants to be in a coalition. However $A$ prefers $C$ than $B$ (as $s_{AC}\ge s_{AB}$) but both $B$ and $C$ prefer to make a coalition with each other (as $s_{BC}\ge s_{AB}$ and $s_{BC}\ge s_{AC}$). Consequently the coalition $BC$ is the unique Nash equilibrium.

\noindent {\it Situation II:} $(\lambda_1,\lambda_2,\lambda_3)=(55,30,15)$:
\Eq{*}{
s_{AB}&=\alpha(\lambda_1,\lambda_2)-(\alpha(\lambda_1)+\alpha(\lambda_2))=-\lambda_2=-30;\\
s_{AC}&=\alpha(\lambda_1,\lambda_3)-(\alpha(\lambda_1)+\alpha(\lambda_3))=-\lambda_1=-15;\\
s_{BC}&=\alpha(\lambda_2,\lambda_3)-(\alpha(\lambda_2)+\alpha(\lambda_3))=0;\\
s_{ABC}&=\alpha(\lambda_1,\lambda_2,\lambda_3)-(\alpha(\lambda_1)+\alpha(\lambda_2)+\alpha(\lambda_3))=0.
}
Then $A$ does not want to make a coalition with either $B$ or $C$ (as the synergy is negative). Similarly neither $B$ nor $C$ wants to make a coalition with $A$ (this essentially follows from the real situation). The coalition $BC$ is irrelevant (which refers to the zero synergy).
\end{exa}

\bigskip

Obviously, as it was announced, the examples above are instrumental to dominant decision systems only. For more complicated decision making systems, we need to define the synergy in a different way.
In general, the synergy depends on the players' decisions. There are essentially two important assertions: 
\begin{enumerate}
 \item zero synergy refers to the situation when aggregation is irrelevant to the rest of the system;
 \item positive synergy should be profitable from the point of view of a decision making system.
\end{enumerate}

In our model, for an effort function $\E\colon \WD \to \R_+$, the $\E$-\emph{synergy} is a function $\sigma_\E\colon \WD \to \R$ defined as follows
\Eq{*}{
\sigma_\E(x,\lambda):=\E(x,\lambda)-(\lambda_1+\cdots+\lambda_n),\qquad n \in \N\text{ and }(x,\lambda) \in D^n \times W_n.
}
In other words, $\sigma_\E$ measures the difference between the given effort and the arithmetic effort, which is the sum of individual efforts. 

If $f\colon D\to X$ is an admissible function and $\E=\beta_f$, then $\sigma_\E$ will simply be denoted as $\sigma_f$. Furthermore, in view of Corollary~\ref{cor:Bbeta}, we can see that the synergy depends only on the mean $\B_f$, that is, the equality $\B_f=\B_g$ implies $\sigma_f=\sigma_g$. Therefore, we can define $\sigma_{\B_f}:=\sigma_f$. 

This property has an important interpretation in the theory of coalitional games. The case when the synergy is negative corresponds to the situation when there appear some distractions in the cooperation (see Example~\ref{ex:toy}.II). The case of positive synergy refers to the situation when the aggregated effort of the group is greater then sum of efforts of the individuals (see Example~\ref{ex:toy}.I). 

\begin{exa}\label{ex:hyperboloid}
Given a manifold $S:=\{(x,y,z)\mid x^2+y^2-z^2=-1 \wedge  z\ge 0\}\subseteq\R^3$ with a parametrization $f \colon \R^2 \to S$ given by $f(x,y)=(x,y,\sqrt{1+x^2+y^2})$. Then $S$ is observable, $f$ is admissible and the Bajraktarevi\'c-type mean $\B_f \colon \mathscr{W}(\R^2)\to \R^2$ is of the following form (here and below $n \in \N$ is fixed, $x,y\in\R^n$, and $\lambda \in W_n$) : 
\Eq{*}{
\B_f((x,y),\lambda)
&=f^{(-1)} \Big( \sum_{i=1}^n \lambda_i f(x_i,y_i) \Big)\\
&=f^{(-1)} \Big( \sum_{i=1}^n \lambda_i x_i,\sum_{i=1}^n \lambda_i y_i,\sum_{i=1}^n \lambda_i \sqrt{1+x_i^2+y_i^2} \Big).
 }
Now define 
\Eq{*}{
\Delta((x,y),\lambda)
:=\Big(\sum_{i=1}^n \lambda_i \sqrt{1+x_i^2+y_i^2}\Big)^2-\Big(\sum_{i=1}^n \lambda_i x_i\Big)^2-\Big(\sum_{i=1}^n \lambda_i y_i\Big)^2.
}
We can easily verify that 
\Eq{*}{
\frac1{\sqrt{\Delta((x,y),\lambda)}} \Big( \sum_{i=1}^n \lambda_i x_i,\sum_{i=1}^n \lambda_i y_i,\sum_{i=1}^n \lambda_i \sqrt{1+x_i^2+y_i^2} \Big) \in S,
 }
which yields
\Eq{*}{
\B_f((x,y),\lambda)
&=\Big( \frac1{\sqrt{\Delta((x,y),\lambda)}}\sum_{i=1}^n \lambda_i x_i,\frac1{\sqrt{\Delta((x,y),\lambda)}}\sum_{i=1}^n \lambda_i y_i \Big);\\
\beta_f((x,y),\lambda)&=\sqrt{\Delta((x,y),\lambda)}.
 }
Furthermore the inverse triangle (Minkowski's) inequality (for $\ell^{1/2}$) applied to the vectors $(\lambda_i^2)$, $(\lambda_i^2x_i^2)$, $(\lambda_i^2y_i^2)$ implies $\beta_f((x,y),\lambda) \ge \lambda_1+\dots+\lambda_n$ or, equivalently, $\sigma_f((x,y),\lambda) \ge 0$.
 
Observe that the above mean is not the standard convex combination of its arguments. Indeed, for the entries $(x,y)=((1,0),(0,1))$ and weights $\lambda=(1,1)$, we get $\B_f((x,y),\lambda)=(\frac{\sqrt6}6,\frac{\sqrt6}6)$, which obviously does not belong to the segment $\conv((0,1),(1,0))$.
\end{exa}

\subsection{Generalized quasi-arithmetic means}
In the next lemma we characterize the subfamily of zero-synergy generalized Bajraktarevi\'c means. Its single variable counterpart was proved in \cite{Pas1910}.

\begin{thm}\label{lem:degen}
Let $f \colon D \to X$ be an admissible function. Then the following conditions are equivalent:
\begin{enumerate}[(i)]
\item $f(D)$ is a convex set and 
\Eq{E:flat}{
\B_f(x,\lambda)=f^{-1} \bigg(\frac{\sum_{i=1}^n \lambda_i f(x_i)} {\sum_{i=1}^n\lambda_i}\bigg)\qquad \text{for all }n \in \N \text{ and }(x,\lambda)\in D^n \times W_n,
}
in particular the right hand side is well-defined for all such pairs;
  \item $\sigma_f\equiv 0$;
 \item \label{BfAss} $\B_f$ is associative, that is, 
 \Eq{iii}{
 \B_f\big((x,y),(\lambda,\mu)\big)
 =\B_f\big(\big(\B_f(x,\lambda),y\big),(\alpha(x,\lambda),\mu) \big)
 }
 for all pairs $(x,\lambda),(y,\mu)\in \WD$ (where $\alpha\colon\WD\to\R_+$ stands for the arithmetic effort function);
 \item Equality \eq{iii} holds for all $(x,\lambda)\in\WD$ and $(y,\mu)\in D\times\R_+$.
\end{enumerate}
\end{thm}

\begin{proof} If $D$ is a singleton, then all of the above conditions are satisfied. Therefore, we may assume that $D$ has at least two distinct elements.  

The implications $(i) \Rightarrow (ii)$, $(ii) \Rightarrow (iii)$, and $(iii) \Rightarrow (iv)$ are easy to check. To prove the implication $(iv) \Rightarrow (i)$, assume that $\B_f$ satisfies \eq{iii} for all $(x,\lambda)\in\WD$ and $(y,\mu)\in D\times\R_+$.

Fix $n \in \N$ and a pair $(x,\lambda) \in D^n \times W_n$. We denote briefly $\bar x:=\B_f(x,\lambda)$, $\bar\lambda:=\beta_f(x,\lambda)$ and $\bar\alpha:=\alpha(x,\lambda)$. Then we have that $\sum_{i=1}^n \lambda_i f(x_i)=\bar\lambda f(\bar x)$.

Now fix $y \in D \setminus\{\bar x\}$ and $\mu>0$. Applying the delegativity of $\B_f$ and condition (iv), we have 
\Eq{*}{
\B_f\big((\bar x,y),(\bar\lambda,\mu)\big)
=\B_f\big((x,y),(\lambda,\mu)\big)
=\B_f\big((\bar x,y),(\bar\alpha,\mu)\big).
}
Consequently $\bar\lambda f(\bar x)+\mu f(y)$ and $\bar\alpha f(\bar x)+\mu f(y)$ are on the same ray, i.e., there exists a constant $C>0$ such that 
\Eq{*}{
C \cdot \big(\bar\lambda f(\bar x)+\mu f(y)\big)
=\bar\alpha f(\bar x)+\mu f(y),
}
which reduces to
\Eq{*}{
0=(C\bar\lambda-\bar\alpha)f(\bar x)+\mu (C-1)f(y).
}
As $y \ne \bar x$, the admissibility implies that $f(y)$ and $f(\bar x)$ are linearly independent. Consequently, the above equality implies $C=1$ and $\bar\lambda=\bar\alpha$. Thus, 
\Eq{*}{
  \frac{\sum_{i=1}^n \lambda_i f(x_i)}{\sum_{i=1}^n \lambda_i}
  =\frac{\sum_{i=1}^n \lambda_i f(x_i)}{\bar\alpha}
  =\frac{\sum_{i=1}^n \lambda_i f(x_i)}{\bar\lambda}
  =f(\bar x)\in f(D),
}
which implies that $f(D)$ is a convex set. Finally, applying $f^{-1}$ side-by-side we get that \eq{E:flat} holds.
\end{proof}

\subsection{Gini means}
We are now going to calculate the sign of the synergy for Gini means. Before we go into the details, we recall a few properties of this family. First, it is easy to observe that $\G_{p,q}=\G_{q,p}$ for all $p,q \in \R$. Furthermore, in a case $q=0$, the Gini mean $\G_{p,0}$ equals the $p$-th Power mean (in particular it is associative).  These means are monotone with respect to their parameters, more precisely, for all $p,q,r,s \in \R$, we have that $\G_{p,q}\le \G_{r,s}$ if and only if $\min(p,q)\le \min(r,s)$ and $\max(p,q)\le \max(r,s)$ (cf.\ \cite{DarLos70}).  Finally, a Gini mean $\G_{p,q}$ is monotone as a mean (in each of its argument) if and only if $pq \le 0$ (see \cite{Los71a,Los71b}). We show that the sign of $pq$ is also important in characterizing the sign of their synergy.

\begin{prop}
   Sign of the synergy of the Gini mean $\G_{p,q}$ coincides with that of $-pq$.
   More precisely, for all $p,q \in \R$, $n\in \N$, nonconstant vector $x\in \R_+^n$ and $\lambda \in \R_+^n$, we have
    $\sign\big(\sigma_{\G_{p,q}}(x,\lambda)\big)=-\sign(pq)$.
\end{prop}

\begin{proof}
 If $pq=0$, then $\G_{p,q}$ is associative and thus Lemma~\ref{lem:degen} implies $\sigma_{\G_{p,q}}\equiv 0$. From now on assume that $pq\ne 0$. 
 Fix $n \in \N$, $\lambda \in \R_+^n$ and nonconstant vector $x\in \R_+^n$. Let 
 \Eq{*}{
 \varphi_p:=\lambda_1x_1^p+\cdots+\lambda_nx_n^p\quad\text{ and }\quad \psi_p:=\lambda_1x_1^p\ln(x_1)+\cdots+\lambda_nx_n^p\ln(x_n)\quad (p \in \R).
} 
 Assume first that $p \ne q$.  As $\gamma_{p,q}=\gamma_{q,p}$, without loss of generality, we can assume that $p>q$. Then by \eq{E:Ginipq}, we have  
 \Eq{*}{
 \gamma_{p,q}(x,\lambda)=\frac{(\lambda_1 x_1^q+\dots+\lambda_n x_n^q)^{\frac{p}{p-q}}}
  {(\lambda_1 x_1^p+\dots+\lambda_n x_n^p)^{\frac{q}{p-q}}}=\bigg(\frac{\varphi_q^p}{\varphi_p^q}\bigg)^{\frac1{p-q}}.
 }
Whence by the definition 
\Eq{*}{
\sigma_{\G_{p,q}}(x,\lambda)=\gamma_{p,q}(x,\lambda)-(\lambda_1+\cdots+\lambda_n)=\bigg(\frac{\varphi_q^p}{\varphi_p^q}\bigg)^{\frac1{p-q}}-\varphi_0.
}
In view of the inequality $p >q$, we obtain
\Eq{*}{
\sigma_{\G_{p,q}}(x,\lambda) > 0 \iff \bigg(\frac{\varphi_q^p}{\varphi_p^q}\bigg)^{\frac1{p-q}}>\varphi_0 \iff \varphi_q^p > \varphi_0^{p-q} \varphi_p^q \iff \bigg(\frac{\varphi_q}{\varphi_0}\bigg)^p > \bigg(\frac{\varphi_p}{\varphi_0}\bigg)^q.
}

For $pq<0$ we obtain that $\sigma_{\G_{p,q}}(x,\lambda) > 0$ is equivalent to $\big(\tfrac{\varphi_q}{\varphi_0}\big)^{1/q}< \big(\tfrac{\varphi_p}{\varphi_0}\big)^{1/p}$. But the last inequality is just the equality between power means. Thus we have $\sigma_{\G_{p,q}}(x,\lambda) > 0$ whenever $pq<0$.

If $pq>0$ then $\sigma_{\G_{p,q}}(x,\lambda) > 0$ is equivalent to $\big(\tfrac{\varphi_q}{\varphi_0}\big)^{1/q} > \big(\tfrac{\varphi_p}{\varphi_0}\big)^{1/p}$. But, as $p>q$ we know that the converse inequality holds. So in this case we obtain $\sigma_{\G_{p,q}}(x,\lambda) < 0$.

In the last case, when $p=q\ne0$, we have
\Eq{*}{
\sigma_{\G_{p,p}}(x,\lambda)=\gamma_{p,p}(x,\lambda)-(\lambda_1+\dots+\lambda_n)=\varphi_p \exp\Big(\frac{-p\psi_p}{\varphi_p}\Big)-\varphi_0.
}
Therefore
\Eq{*}{
\sigma_{\G_{p,p}}(x,\lambda) < 0 \iff \varphi_p \exp\Big(\frac{-p\psi_p}{\varphi_p}\Big)-\varphi_0<0 \iff \exp\Big(\frac{-p\psi_p}{\varphi_p}\Big) < \frac{\varphi_0}{\varphi_p}.
}
We can apply the strictly decreasing mapping $\R_+ \ni \xi \mapsto \sign(p)\xi^{-1/p}$ side-by-side to obtain
\Eq{*}{
\sigma_{\G_{p,p}}(x,\lambda) < 0 &\iff \sign(p) \exp\Big(\frac{\psi_p}{\varphi_p}\Big) > \sign(p) \Big(\frac{\varphi_p}{\varphi_0}\Big)^{1/p}\\
&\iff \sign(p) \G_{p,p}(x,\lambda) >\sign(p) \G_{p,0}(x,\lambda).
}
But the inequality on the right-hand-side holds for all $p\in \R \setminus\{0\}$, which completes the proof.
\end{proof}

\begin{xrem}
 Observe that Gini mean has a positive synergy if the graph of $\gamma_{p,q}$ is hyperbolic and negative for parabolic graphs. In the case of hyperboloid in Example~\ref{ex:hyperboloid} the synergy was also positive.  
\end{xrem}


\end{document}